\documentclass[12pt]{amsart}
\usepackage[top=1in, bottom=1in, left=1in, right=1in]{geometry}
\usepackage{dsfont}
\usepackage{amsmath}
\usepackage{amsthm}
\usepackage{amssymb}
\usepackage{multirow}
\usepackage{mathtools}
\usepackage{times}
\usepackage[utf8]{inputenc}
\usepackage[english]{babel}

\usepackage{indentfirst}  
\usepackage{graphicx}
\usepackage{color}
\usepackage[colorlinks=true]{hyperref}
\hypersetup{colorlinks=true, citecolor=green, linkcolor=green, filecolor=magenta, urlcolor=cyan}
\usepackage[shortlabels]{enumitem}
\usepackage{titlesec}
\titleformat{\chapter}[display]{\bfseries\huge}{\filright\huge\chaptertitlename~\thechapter}{3ex}{\titlerule\vspace{1ex}\filright}[\vspace{1ex}\titlerule]

\newtheorem{thm}{Theorem}[section]
\newtheorem{prop}[thm]{Proposition}
\newtheorem{lem}[thm]{Lemma}
\newtheorem{cor}[thm]{Corollary}
\newtheorem{conjecture}[thm]{Conjecture}
\theoremstyle{definition}

\theoremstyle{definition}

\theoremstyle{remark}
\newtheorem{remark}[thm]{Remark}

\newenvironment{feqn*}{\begin{mdframed}\begin{equation*}}{\vspace{1mm}
\end{equation*}\end{mdframed}}

\numberwithin{equation}{section}

\newcommand{\N}{\mathbb{N}}

\newcommand{\CA}{\mathcal{A}}

\newcommand{\CB}{\mathcal{B}}

\newcommand{\CD}{\mathcal{D}}

\DeclareMathOperator{\supp}{supp}

\newcommand{\bs}\boldsymbol{}
\renewcommand{\geq}{\geqslant}
\renewcommand{\leq}{\leqslant}

\renewcommand{\mod}[1]{\,({\rm mod}\,#1)}
\newcommand{\vect}[1]{\overrightarrow{\boldsymbol{#1}}}

\definecolor{blue}{rgb}{.2,.6,.75}
\definecolor{green}{rgb}{.4,.7,.4}
\definecolor{red}{rgb}{1,0,0}

\titleformat{\section}[block]{\scshape\centering}{\arabic{section}.}{1ex}{}{}

\begin{document}




\title[Right-angled triangles with almost prime hypotenuse]{Right-angled triangles with almost prime hypotenuse}

\author{Cihan Sabuncu}

\address{D\'epartement de math\'ematiques et de statistique\\
Universit\'e de Montr\'eal\\
CP 6128 succ. Centre-Ville\\
Montr\'eal, QC H3C 3J7\\
Canada}


\email{cihan.sabuncu@umontreal.ca}

\subjclass[2010]{}

\date{\today}

\begin{abstract}
The sequence OEIS A281505 consists of distinct odd legs in right triangles with integer sides and prime hypotenuse. In this paper, we count the closely related quantity of even legs with almost prime hypotenuse. More precisely, we obtain the correct order of magnitude upper and lower bounds for the set of distinct even legs with 5-almost prime hypotenuse. This is a strong version of the appropriate analogy to a conjecture of Chow and Pomerance (stated there for prime hypotenuse).
\end{abstract}

\maketitle

\section{Introduction}
Given a right-angled triangle we can parametrize its sides by $a^2-b^2 , 2ab, a^2+b^2$. The sequence OEIS A281505 counts the odd legs with prime hypotenuse $\{ n \leq N : \exists (a,b)\in \N^2 , n=a^2-b^2, a^2+b^2 \text{ prime, } 0<b<a \}$. In the same manner, we can study the set of even legs with prime hypotenuse $\CA'(N):=\{ n\leq N : \exists (a,b)\in \N^2, n=2ab, a^2+b^2 \text{ prime} \}$. We note that the condition $n=2ab$ can be changed to $n=ab$ by working with $\CA'(2N)$ instead. Therefore, we write $\CA(N):= \CA'(2N)$. A result of Chow and Pomerance \cite{MR3709664} shows
$$ \frac{N}{(\log N)^{c+o(1)}} \ll \#\CA(N) \ll \frac{N(\log\log N)^{O(1)}}{(\log N)^{\delta}} $$
where $\delta :=1-\frac{1+\log\log 2}{\log 2}=0.08607\cdots$ is the Erd\H{o}s-Tenenbaum-Ford constant appearing in the multiplication table problem, and $c=\log 4-1=0.38629\cdots$. Moreover, they conjecture the correct order of magnitude to be that of the upper bound. In this paper, we get sharp upper and lower bounds for the related set of even legs with almost prime hypotenuse,
\begin{equation*}
\CB(N)=\{n\leq N : \exists (a,b)\in \N^2, n=ab, \Omega(a^2+b^2)\leq 5 , P^-(a^2+b^2)>N^{1/9}\}
\end{equation*}
where $P^-(n)$ is the smallest prime factor of $n$. Our result agrees with their conjecture and also gives the power of the double-logarithmic factor.
\begin{thm}\label{main_theorem}
Let $N$ be a large number. We have
$$ \#\CB(N) \asymp \frac{N}{(\log N)^{\delta}\sqrt{\log\log N}}. $$
\end{thm}

\begin{remark}
The condition $P^-(a^2+b^2)>N^{1/9}$ can be changed to $P^-(a^2+b^2)>N^{1/8+\eta}$ for any fixed $\eta >0$ with more bookkeeping on the exponents.
\end{remark}
\begin{remark}
Assuming the Elliott-Halberstam \cite{MR0276195} conjecture, we can get the Theorem \ref{main_theorem} with $\Omega(a^2+b^2)\leq 3$.
\end{remark}
\noindent For a typical $ab\leq N$, we guess $a^2+b^2\approx N$. Thus, for $a^2+b^2$ to be a prime, we need $P^-(a^2+b^2)>\sqrt{N}$. This should behave similar to $P^-(a^2+b^2)>N^{1/9}$. Thus, $\CA(N)$ should behave like $\CB(N)$. In accordance with this heuristic, we state our conjecture below.
\begin{conjecture}\label{main_conjecture}
Let $N$ be a large number. We have
$$ \#\CA(N) \asymp \frac{N}{(\log N)^{\delta} \sqrt{\log\log N}} .$$
\end{conjecture}
\begin{remark}
The upper bound in Conjecture \ref{main_conjecture} follows from the upper bound in Theorem \ref{main_theorem}. Moreover, to get the lower bound we use the second moment method (see \eqref{second_moment_method}). So, we study the average and the second moment of $ r^*(n) := \#\{ (a,b)\in \N : n=ab, a^2+b^2 \text{ prime} \}$ on the subset $\{ n\leq N : \omega(n)\leq \frac{\log\log N}{\log 2} \}$ . We need a lower bound for the average and an upper bound for the second moment. The upper bound for the second moment follows from Theorem \ref{what_we_want_to_prove} below. Thus, to get the lower bound in Conjecture \ref{main_conjecture} we need
$$ \sum_{\substack{n\leq N \\ \omega(n)\leq \frac{\log\log N}{\log 2}}}r^*(n) \gg \frac{N}{(\log N)^{\delta} \sqrt{\log\log N}}. $$
\end{remark}
\subsection*{Idea of the proof:}
We define our representation function
$$ r(n)= \#\{ (a,b)\in \N^2 : n=ab, \Omega(a^2+b^2)\leq 5 , P^-(a^2+b^2)>N^{1/9} \}. $$
Then we have
\begin{equation} \label{second_moment_method}
 (\supp r\cap [1,N]) = \#\CB(N) \geq \bigg(\sum_{\substack{n\leq N \\ n\in\CD }} r(n) \bigg)^2 \bigg/ \bigg( \sum_{\substack{n\leq N \\ n\in \CD}} r^2(n) \bigg),
\end{equation}
by the Cauchy-Schwarz inequality, for any $\CD\subset \N$. We want to maximize this ratio by taking a set $\CD$ such that the second moment is on the same order as the mean. We notice that if $r(n)\ll 1$, then 
$$\sum_{\substack{n\leq N \\ n\in \CD}}r(n)\leq \sum_{\substack{n\leq N \\ n\in \CD}} r^2(n) \ll \sum_{\substack{n\leq N \\ n\in \CD}} r(n).$$
Thus, we want a set $\CD \subset \N$ where $r(n)$ is almost constant. We guess that $r(n)\approx \tau(n)/(\log N^{1/9})$, so we want $\tau(n)\ll \log N^{1/9}$. For a square-free $n$ we have $\tau(n) = 2^{\omega(n)}\ll \log N^{1/9}$. This is integers with $\omega(n) \leq \frac{\log\log N}{\log 2} + O(1)$. Hence, we take
$$ \CD := \{ n\in \N : \omega(n)\leq K \} \text{ where } K:=\bigg\lfloor \frac{\log\log N}{\log 2} \bigg\rfloor. $$
Moreover, we have
$$\#\CB(N) \leq \sum_{\substack{n\leq N \\ \omega(n)\leq K}} r(n) + \#\{ n\leq N : \omega(n)>K \}. $$
We use the Hardy-Ramanujan theorem to bound the size of the set appearing here. Hence, Theorem \ref{main_theorem} will follow from Lemma \ref{omega_large}, and Theorem \ref{what_we_want_to_prove} below.
\begin{thm}\label{what_we_want_to_prove}
Let $N$ be a large number. We have
\begin{align*}
\sum_{\substack{n\leq N \\ \omega(n)\leq K}} r(n) &\asymp \frac{N}{(\log N)^{\delta} \sqrt{\log\log N}}, \\
\sum_{\substack{n\leq N \\ \omega(n) \leq K}} r^2(n) &\ll \frac{N}{(\log N)^{\delta} \sqrt{\log\log N}}.
\end{align*}
\end{thm}
\noindent We use the technique developed in Shiu \cite{MR0552470} to prove the upper bounds in Section \ref{section_upper_bounds}. We will prove the lower bound of the average of $r(n)$ in Section \ref{section_lower_bound} using a weighted sieve result of Richert (see \cite{MR0424730}[Theorem 9.3]) in the subset of square-free integers $n\leq N$ with exactly $K$ prime factors. \\
\subsection*{Notation:} We will use the standard asymptotic notation. $N$ will be a large integer and $K:=\lfloor \frac{\log\log N}{\log 2} \rfloor$ like above. We will also define a cut off variable;
\begin{equation} \label{cut_off_variable}
C(N) := N^{\log\log\log\log N/\log\log N}. 
\end{equation}
We will use $p$ for primes, $m,n,a,b,d,h,k,u,v$ with or without subscripts for integers. $\omega(n)$ will be the number of prime factors of $n$, and $\tau(n)$ will be the divisor function. We will also use $\tau_k(n)=\#\{d_1d_2\cdots d_k=n\}$ for the k-divisor function. $P^+(n)$ will be the largest prime factor of $n$ and $P^-(n)$ will be the smallest prime factor of $n$.

\subsection*{Acknowledgements} The author is grateful to Andrew Granville for his continued guidance and suggestions. He wishes to thank Dimitris Koukoulopoulos for helpful discussions. He would also like to thank Tony Haddad, Sun-Kai Leung, Stelios Sachpazis and Christian T\'afula for their comments.

\section{Preliminary Lemmas}

\begin{lem}\label{omega_large}
Let $N$ be a large number. We have
$$ \#\{ n\leq N : \omega(n)>K \} \ll \frac{N}{(\log N)^{\delta}\sqrt{\log\log N}}. $$
\end{lem}
\begin{proof}
We let $k\geq K$, then we have by the Hardy-Ramanujan theorem,
$$ \sum_{\substack{n\leq N \\ \omega(n)=k}} 1 \ll \frac{N}{\log N} \frac{(\log\log N + O(1))^{k-1}}{(k-1)!}. $$
Now, we sum over $k\geq K$, which gives us
\begin{align*}
\sum_{\substack{n\leq N \\ \omega(n)\geq K}}1 &\ll \frac{N}{\log N} \sum_{k\geq K}  \frac{(\log\log N + O(1))^{k-1}}{(k-1)!} \\
&\ll \frac{N}{\log N} \frac{(\log\log N + O(1))^K}{K!} \asymp \frac{N}{(\log N)^{\delta} \sqrt{\log\log N}},
\end{align*}
by Stirling's approximation $m!\sim \sqrt{2\pi m}(m/e)^m$.
\end{proof}

\begin{lem}\label{beta_sieve_result}
Let $M$ and $N$ be large numbers with $M > N$ and $a,b\in \N$ with $2|ab$. Then
$$ \#\{ n \leq M : P^-(n(a^2+n^2b^2)) >N \} \ll \frac{M}{(\log N)^2} \mathfrak{S}(a,b) ,$$
where
$$ \mathfrak{S}(a,b) = \prod_{p\nmid ab}\bigg( 1 - \frac{2+\chi_4(p)}{p}\bigg)\bigg(1 - \frac{1}{p}\bigg)^{-2} \prod_{p|ab} \bigg(1 - \frac{1}{p} \bigg)^{-1}, $$
and $\chi_4$ is the non-principal Dirichlet character $\bmod 4$. Moreover,
\begin{equation}\label{first_singular_series_bound}
\mathfrak{S}(a,b)\ll \sum_{r|ab}\frac{\mu^2(r)2^{\omega(r)}}{r} .
\end{equation}
\end{lem}
\begin{proof}
This is an application of the Selberg sieve (see \cite{MR2647984}[Theorem 7.14]). We first note that
$$ \#\{ n \leq M : P^-(n(a^2+n^2b^2)) >N \} \leq \#\{ n \leq M : P^-(n(a^2+n^2b^2)) >N^{\delta} \},  $$
where we will choose $0<\delta <1$ later. We will bound the set on the right. Also note that we can assume $2$ can't divide both $a$ and $b$, otherwise this set is empty and we have the upper bound trivially. \\
To start, we write
\begin{equation}\label{splitting_idea}
\sum_{\substack{n\leq M \\ d|n(a^2+n^2b^2)}} 1 = \sum_{\substack{m \mod d \\ m(a^2+m^2b^2)\equiv 0 \mod d}} \sum_{\substack{n\leq M \\ n\equiv m \mod d}} 1 = \frac{\nu_d}{d} M + O(\nu_d) ,
\end{equation}
where
$$ \nu_d : = \#\{ m\mod d : m(a^2+m^2b^2)\equiv 0 \mod d \}. $$
Then we have $\nu_d \geq 1$ as $m\equiv 0 \mod d$ is always a solution. Note also that $\nu_2=1$ since $2$ divides exactly one of $a$ or $b$. We just need to see for $2\leq v \leq w$
$$ \sum_{v\leq p \leq w} \frac{\nu_p \log p}{p} \leq \sum_{\substack{v\leq p \leq w}} \frac{3\log p}{p} \ll \log \frac{2w}{v} . $$
Then we choose $D=N$ and we want to bound 
$$ J(D)=\sum_{d|P(N^{\delta}) , d<D} h(d) = \prod_{p<N^{\delta}} (1 + h(p)) + \sum_{d|P(N^{\delta}), d\geq D} h(d),$$
where $P(N)=\prod_{p\leq N}p$, and $h(p)=(1-\nu_p/p)^{-1}\nu_p/p$, which is well defined as $\nu_2=1$. We bound the second sum,
\begin{align*}
\sum_{d|P(N^{\delta}), d\geq D} h(d) \ll \frac{e^{O(\delta^{-1})}}{(\delta^{-1}\log \delta^{-1})^{\delta^{-1}}} \exp\bigg( \sum_{p\leq N^{\delta}} h(p) \bigg) \ll \frac{e^{O(\delta^{-1})}}{(\delta^{-1}\log \delta^{-1})^{\delta^{-1}}} \prod_{p<N^{\delta}} (1 + h(p)),
\end{align*}
where we use Koukoulopoulos \cite{MR3971232}[Theorem 16.3] to get the first inequality. For the second inequality, we use $\log(1+x)\geq x-\frac{x^2}{2}$ for $0\leq x \leq 1$, and use $h(p)\ll 1/p$ to bound the term $\sum_{p\leq N^{\delta}} \frac{h(p)^2}{2} \ll 1$. We choose $\delta<1$ small enough so that this sum is $< \frac{1}{2} \prod_{p<N^{\delta}} (1 + h(p))$. Thus, we get $J(D) \gg \prod_{p<N^{\delta}} (1 + h(p)) = \prod_{p<N^{\delta}}(1-\nu_p/p)^{-1} \asymp (\log N)^2 \prod_{p\geq 2}(1-\nu_p/p)^{-1}(1-1/p)^2$ by Mertens' theorem. \\
For \eqref{first_singular_series_bound}, we complete the product over $p\nmid ab$ to get
$$ \mathfrak{S}(a,b)\asymp \prod_{p|ab} \bigg( 1 - \frac{2+\chi_4(p)}{p}\bigg)^{-1} \bigg( 1  - \frac{1}{p}\bigg) \asymp \prod_{p|ab} \bigg(1 + \frac{1+\chi_4(p)}{p} \bigg) \ll \sum_{r|ab}\frac{\mu^2(r)2^{\omega(r)}}{r},$$
where we use $\prod_{p|ab}(1-\frac{2+\chi_4(p)}{p})^{-1}(1-\frac{1}{p})(1+\frac{1+\chi_4(p)}{p})^{-1} \asymp 1$ since it is an absolutely convergent product.
\end{proof}

\begin{lem}\label{second_beta_sieve_result}
Let $M$ and $N$ be large numbers with $M>N$ and $a_1,a_2,a_3,a_4,a_5\in \N$ with $2|a_1a_2a_3a_4a_5$. Then
$$ \#\{ n \leq M : P^-(n(a_1^2a_2^2+n^2a_3^2a_4^2a_5^2)(a_1^2a_4^2 n^2 + a_2^2 a_3^2a_5^2)) >N \} \ll \frac{M}{(\log N)^3} \mathfrak{S}(\vect{a}) ,$$
where
$$ \mathfrak{S}(\vect{a}) = \prod_{p\nmid a_1a_2a_3a_4a_5}\bigg( 1 - \frac{3+2\chi_4(p)}{p}\bigg)\bigg(1 - \frac{1}{p}\bigg)^{-3} \prod_{p|a_1a_2a_3a_4a_5} \bigg(1 - \frac{2+\chi_4(p)}{p}\bigg)\bigg(1 - \frac{1}{p} \bigg)^{-3}, $$
where $\chi_4$ is the non-principal Dirichlet character $\bmod 4$. Moreover,
\begin{equation}\label{more_variable_singular_series_bound}
\mathfrak{S}(\vect{a})\ll \sum_{r|a_1a_2a_3a_4a_5}\frac{\mu^2(r)2^{\omega(r)}}{r} .
\end{equation}
\end{lem}
\begin{proof}
The proof of this is similar to Lemma \ref{beta_sieve_result}, we just need to change the $\nu_d$. So, we just prove \eqref{more_variable_singular_series_bound}. We complete the product over $p\nmid a_1a_2a_3a_4a_5$ to get
$$ \mathfrak{S}(\vect{a})\asymp \prod_{p|a_1a_2a_3a_4a_5} \bigg(1 - \frac{2+\chi_4(p)}{p} \bigg) \bigg(1- \frac{3+2 \chi_4(p)}{p} \bigg)^{-1} \asymp \prod_{p|a_1a_2a_3a_4a_5} \bigg(1 + \frac{1+\chi_4(p)}{p} \bigg) ,$$
where we similarly got rid of high order terms, and can turn the product into the sum to get \eqref{more_variable_singular_series_bound}.
\end{proof}

\begin{lem}\label{weighted_sieve_result}
Let $N$ be a large number, and $a,b\in \N$ with $\gcd(a,b)=1$, and $ab\leq N^{1/1000}$. Then
$$ \#\{ p\leq N : \Omega(a^2+p^2b^2)\leq 5 , P^-(a^2+p^2b^2)>(N/\log N)^{1/8} \} \gg \frac{N}{(\log N)^2} \mathfrak{S}'(a,b), $$
where
$$ \mathfrak{S}'(a,b)=\prod_{p\nmid ab} \bigg(1 - \frac{1+\chi_4(p)}{p-1} \bigg) \bigg(1 - \frac{1}{p} \bigg)^{-1} \prod_{p| ab}\bigg(1 - \frac{1}{p} \bigg)^{-1}. $$
\end{lem}
\begin{proof}
This is an application of the Richert sieve (see \cite{MR0424730}[Theorem 9.3]). We have to check the required properties. First note that similar to \eqref{splitting_idea},
$$ \sum_{\substack{p\leq N \\ d|a^2+p^2b^2}}1 = \frac{\nu_d}{\varphi(d)} \cdot \frac{N}{\log N} + O( \tau(d) E(N,d)),$$
where
$$ E(N,d):= \max_{(u,d)=1}\left| \sum_{\substack{p\leq N \\ p\equiv u \mod d}} 1 - \frac{1}{\varphi(d)} \sum_{p\leq N}1\right| ,$$
and
$$ \nu_d := \# \{ m \mod d : (m,d)=1, a^2+m^2b^2\equiv 0 \mod d \}, $$
and $\nu_d\leq \tau(d)$. Next, we have
\begin{align*}
\sum_{d\leq \frac{\sqrt{N}}{(\log N)^{100}}} &\mu^2(d) 3^{\omega(d)} \tau(d) E(N,d) \leq \sqrt{N} \sum_{d\leq \frac{\sqrt{N}}{(\log N)^{100}}} \frac{\mu^2(d) 3^{\omega(d)}}{\sqrt{\varphi(d)}}  \sqrt{E(N,d)}\\
&\leq \sqrt{N} \bigg(\sum_{d\leq \frac{\sqrt{N}}{(\log N)^{100}}} \frac{\mu^2(d) 9^{\omega(d)} \tau^2(d)}{\varphi(d)}\bigg)^{1/2} \bigg(\sum_{d\leq \frac{\sqrt{N}}{(\log N)^{100}}} E(N,d)\bigg)^{1/2} \\
&\ll \frac{N}{(\log N)^2} ,
\end{align*}
by an application of the Cauchy-Schwartz inequality and Bombieri-Vinogradov theorem (see \cite{MR3971232}[Theorem 18.9]). We also check for $2\leq v \leq w \leq N$
\begin{align*}
\log \frac{w}{v} + O(\log\log N) =\sum_{v\leq p \leq w} \frac{(1+\chi_4(p))\log p}{p} - 2 \sum_{p|ab} \frac{\log p}{p} \leq \sum_{v\leq p \leq w} \frac{\nu_p\log p}{p} &\leq  \sum_{\substack{v\leq p \leq w \\ p\equiv 1 \mod 4}} \frac{2\log p}{p}  \\
&= \log \frac{w}{v} + O(1)
\end{align*}
since $\nu_p=1+\chi_4(p)$ if $p\nmid ab$, and $\nu_p=0$ if $p|ab$.
And lastly, we need to check for $2\leq z \leq y \leq N$ \footnote{We actually need the bound $\ll N/z$, but as stated in the footnote of \cite{MR0424730}[$(\Omega_3)$, p. 253] we can get the same result with this weaker condition.}
$$ \sum_{z\leq q< y} \sum_{\substack{p\leq N \\ q^2 | a^2+p^2b^2}} 1 \ll N \sum_{z\leq q < y} \frac{\tau(q)}{q^2}\ll \frac{N\log N}{z} $$
where we split into congruence classes like \eqref{splitting_idea}, and we count integers instead of primes.\\
\noindent Now, since $ab\leq N^{1/1000}$, we have,
$$ a^2+p^2b^2 \leq \left(\frac{N}{\log N}\right)^{2+1/500}. $$
So, using the level of distribution $1/2$ coming from Bombieri-Vinogradov theorem, we have $4+1/250 \leq \Lambda_r$ where
$$\Lambda_r := r+1- \frac{\log 4/(1+3^{-r})}{\log 3}.$$
This gives us $r=5$ as a choice and completes the proof of the lemma.
\end{proof}

\begin{lem}\label{fixed_prime_factor_smooth_number_calculation}
Let $\eta >0$ be a small number, and $x\geq y \gg_{\eta} 1 $, and $u=\log x /\log y$. Also let $\ell\in \N$, and take $k\in \N$. We have
$$ \sum_{\substack{d> x \\ P^+(d)\leq y \\ \omega(d)=k}} \frac{\tau_{\ell}(d)}{d} \leq e^{- cu}\frac{(\ell \cdot \log\log x+ O_{\eta}(1))^k}{k!}, $$
for some constant $c>0$.
\end{lem}
\begin{proof}
We use Rankin's trick to get
$$ \sum_{\substack{d> x \\ P^+(d)\leq y \\ \omega(d)=k}}\frac{\tau_{\ell}(d)}{d} \leq x^{-\phi} \sum_{\substack{P^+(d)\leq y \\ \omega(d)=k}} \frac{\tau_{\ell}(d)}{d^{1-\phi}} $$
for some $0< \phi <  1 /2 - \eta$, and we write $\phi=\varepsilon/\log y$. Now, we use Erd\H{o}s' trick to get
$$ \sum_{\substack{P^+(d)\leq y \\ \omega(d)=k}} \frac{\tau_{\ell}(d)}{d^{1-\phi}} \leq \frac{1}{k!} \bigg( \sum_{p^{\nu} \leq y} \frac{\tau_{\ell}(p^{\nu})}{p^{\nu (1-\phi)}} \bigg)^k = \frac{1}{k!} \bigg( \sum_{p\leq y} \frac{\ell}{p^{1-\phi}} + O_{\eta}(1) \bigg)^k . $$
For the range $p\leq y^{1/\varepsilon}$, we have $p^{\varepsilon/\log y} = 1 + O(\varepsilon \log p /\log y)$. Thus
\begin{align*}
\sum_{p\leq y^{1/\varepsilon}} \frac{\ell}{p^{1-\varepsilon/\log y}} &= \sum_{p\leq y^{1/\varepsilon}} \frac{\ell }{p} + O\bigg( \frac{\varepsilon}{\log y} \sum_{p\leq y^{1/\varepsilon}} \frac{\log p}{p}\bigg) \\
&\leq \ell \cdot \log\log y + O(1).
\end{align*}
For the remaining range $y^{1/\varepsilon} < p \leq y$, we use partial summation.
\begin{align*}
\sum_{y^{1/\varepsilon} < p \leq y} \frac{\ell}{p^{1-\varepsilon /\log y}}  \ll 1 + \frac{e^{\varepsilon}}{\log y} + \int_{y^{1/\varepsilon}}^y \frac{t^{\varepsilon /\log y}}{t\log t} dt \ll \frac{e^{\varepsilon}}{\varepsilon}.
\end{align*}
Putting these together we get
$$ \sum_{\substack{d> x \\ P^+(d)\leq y \\ \omega(d)=k}} \frac{\tau_{\ell}(d)}{d} \leq \frac{e^{-u \varepsilon}}{k!} \bigg(\ell \cdot \log\log y + O_{\eta}\bigg(\frac{e^{\varepsilon}}{\varepsilon} \bigg) \bigg)^k .$$
Now, the result follows if we take $\varepsilon>0$ to be some small constant and bounding $\log\log y\leq \log\log x$.
\end{proof}
\begin{lem}\label{almost_prime_on_divisor_function}
Let $x$ be a large number, and $k < 2\log\log x$ be a positive integer. We have
$$ \sum_{\substack{n\leq x \\ 2 \nmid n \\ \omega(n)=k}} \frac{\mu^2(n) \tau(n)}{n} = \frac{(2\log\log x + O(1))^k}{k!} . $$
\end{lem}
\begin{proof}
We first prove the upper bound. By Erd\H{o}s' trick we have
$$\sum_{\substack{n\leq x \\ 2\nmid n \\ \omega(n)=k}} \frac{\mu^2(n) \tau(n)}{n} \leq \frac{1}{k!} \bigg( \sum_{2<p\leq x} \frac{2}{p} \bigg)= \frac{(2\log\log x +O(1))^k}{k!}.$$
For the lower bound, we use inclusion-exclusion to get
\begin{align*}
\sum_{\substack{n\leq x \\ 2\nmid n \\ \omega(n)=k}} \frac{\mu^2(n) \tau(n)}{n} &\geq \frac{1}{k!} \bigg(\sum_{2<p\leq x} \frac{2}{p} \bigg)^k - \frac{\binom{k}{2}}{k!} \bigg( \sum_{p\leq x} \frac{2}{p} \bigg)^{k-2} \bigg( \sum_{p\leq x} \frac{4}{p^2} \bigg) \\
&=\frac{(2\log\log x + O(1))^k}{k!}.
\end{align*}
Hence, we get the result.
\end{proof}
\section{Upper Bounds}\label{section_upper_bounds}
\subsection*{Upper bound for the average of \texorpdfstring{$r(n)$}{r(n)}:} We use an idea due to Shiu \cite{MR0552470} (see \cite{MR3971232}[Theorem 20.3] for a more recent exposition.). We want to understand
$$ \sum_{\substack{n\leq N \\ \omega(n)\leq K }} r(n) \leq \sum_{\substack{ab\leq N \\ P^-(a^2+b^2)>N^{1/9} \\ \omega(ab)\leq K}}1 .$$
Assume without loss of generality, $a\leq b$, then we take $b=p_1^{e_1} p_2^{e_2} \cdots p_{\omega(b)}^{e_{\omega(b)}}$ for primes $p_1<p_2<\cdots < p_{\omega(b)}$, and $e_i\in \N$ for $i\leq \omega(b)$. Then we let $d=p_1^{e_1} \cdots p_{\ell}^{e_{\ell}}$ such that $d\leq N^{3/8} < d p_{\ell+1}^{e_{\ell+1}}$, and $m=b/d$, $P^-(m)=p_{\ell+1}$. If $b\leq N^{3/8}$, that is $m=1$, then we have trivially a contribution of $\ll N^{3/4}$ which we can discard. So, we can assume $m>1$ and hence $m>P^-(m)$. Thus, we can bound the average by the sum over $a\leq \sqrt{N}$ and the sum over $dm\leq N/a$ with $P^-(a^2+d^2m^2)>N^{1/9}$ and $P^-(m) > \max \{ P^+(d), (N^{3/8}/d)^{1/e^{\ell+1}} \}$. Now, we split into three cases according to the sizes of $P^-(m)$ and $d$. \\
\textit{Case of $P^-(m)>N^{3/16}$:} We have in this case that $m\leq N$ and $P^-(m)>N^{3/16}>N^{1/9}$, hence $\omega(m)\leq 16/3$. Thus, it is bounded, and since $\omega(adm)\leq K$, the prime factors concentrates around $ad$. So we can bound by
\begin{align}
\leq 2 \sum_{a\leq \sqrt{N}} \sum_{\substack{d\leq N^{3/8} \\ \omega(ad)\leq K}} \sum_{\substack{m\leq N/da \\ P^-(m(a^2+d^2m^2))>N^{1/9}}} 1 &\ll \frac{N}{(\log N)^2} \sum_{a\leq \sqrt{N}} \sum_{\substack{d\leq N^{3/8} \\ \omega(ad)\leq K}} \frac{\mathfrak{S}(a,d)}{ad} \nonumber \\
&\ll \frac{N}{(\log N)^2}  \sum_{r} \frac{\mu^2(r)2^{\omega(r)}}{r} \sum_{a\leq \sqrt{N}} \frac{1}{a} \sum_{\substack{d\leq N^{3/8} \\ r|ad \\ \omega(ad)\leq K}} \frac{1}{d} \nonumber \\
&\leq \frac{N}{(\log N)^2}  \sum_{r} \frac{\mu^2(r)2^{\omega(r)}}{r} \sum_{\substack{h\leq N^{7/8} \\ r|h \\ \omega(h)\leq K }} \frac{\tau(h)}{h}, \label{Case_1_bound}
\end{align}
where we used Lemma \ref{beta_sieve_result} for the first line, \eqref{first_singular_series_bound} for the second line, and $h=ad$ for the last line. Now, we split $h=rk$, then
\begin{align*}
\sum_{\substack{h\leq N^{7/8} \\ r|h \\ \omega(h)\leq K }} \frac{\tau(h)}{h} &\leq \frac{\tau(r)}{r}\sum_{\substack{k\leq N^{7/8} \\ \omega(k)\leq K}} \frac{\tau(k)}{k} \\
&\leq \frac{\tau(r)}{r} \frac{1}{K!}\bigg( \sum_{p\leq N^{7/8}} \frac{2}{p} +O(1) \bigg)^K \\
&\ll \frac{\tau(r)}{r} \cdot \frac{(\log N)^{2-\delta}}{\sqrt{\log\log N}},
\end{align*}
by Erd\H{o}s' trick, and Stirling's approximation with $K=\lfloor \frac{\log\log N}{\log 2} \rfloor$. So, we get the upper bound in this case by \eqref{Case_1_bound} since the sum over $r$ converges. \\
\textit{Case of $P^-(m)\leq N^{3/16}$ and $d>N^{3/16}$:} Let $z=N^{3/8}$. We do a dyadic decomposition $z_j \geq P^+(d) > z_{j+1}$ with $z_j=z^{2^{-j}}$ for $j\geq 1$, and we choose $J$ such that $z_{J+1} \leq C(N) < z_J$ where $C(N)$ is given in \eqref{cut_off_variable}. Thus, we have $P^-(m)>z_{j+1}$ and $P^-(a^2+m^2d^2)>N^{1/9}$ can be bounded by $P^-(a^2+m^2d^2)>z_{j+1}$ since $z_{j+1} \leq z_2 \leq N^{1/9}$. Notice also that since $P^-(m)>z^{2^{-j-1}}$, we have $\omega(m)\leq 2^{j+4}/3\leq 2^{j+3}$, and since $(ad,m)=1$, we have $\omega(ad)\leq K- 2^{j+3}$. We also note that $2^{j+3} = o(\log\log N)$ for $j\leq J$. Then we get for $j\leq J$
\begin{align}
&\sum_{j\leq J} \sum_{a\leq \sqrt{N}} \sum_{\substack{N^{3/16}\leq d\leq N^{3/8} \\ P^+(d)\leq z_{j+1} \\ \omega(ad)\leq K-2^{j+3}}} \sum_{\substack{m\leq N/ad \\ P^-(m(a^2+m^2d^2))>z_{j+1}}} 1 \ll \frac{N}{(\log N)^2} \sum_{j\leq J} 4^j \sum_{a\leq \sqrt{N}} \sum_{\substack{N^{3/16}\leq d\leq N^{3/8} \\ P^+(d)\leq z_{j+1} \\ \omega(ad)\leq K-2^{j+3}}} \frac{\mathfrak{S}(a,d)}{ad} \nonumber \\
&\ll \frac{N}{(\log N)^2}  \sum_{j \leq J} 4^j \sum_{r} \frac{\mu^2(r)2^{\omega(r)}}{r} \sum_{a\leq \sqrt{N}} \frac{1}{a} \sum_{\substack{N^{3/16}\leq d\leq N^{3/8} \\ P^+(d)\leq z_{j+1} \\ r|ad \\ \omega(ad)\leq K-2^{j+3}}} \frac{1}{d} \nonumber \\
&\leq  \frac{N}{(\log N)^2}  \sum_{j\leq J} 4^j \sum_{r} \frac{\mu^2(r)2^{\omega(r)}}{r^2}\sum_{i+\ell \leq K- 2^{j+3}} \sum_{\substack{a\leq \sqrt{N}/r \\ \omega(a)=i}} \frac{1}{a} \sum_{\substack{N^{3/16}/r\leq d\leq N^{3/8}/r \\ P^+(d)\leq z_{j+1} \\ \omega(d)=\ell}} \frac{1}{d}, \label{equation_for_case_2}
\end{align}
where we split $\omega(ad)=\omega(a)+\omega(d)$ because $\gcd(a,d)=1$. Now, we first notice that we can bound the contribution of $r\geq N^{1/16}$ trivially,
$$  \frac{N}{(\log N)^2}   \sum_{j\leq J} 4^j\sum_{r\geq N^{1/16}} \frac{\mu^2(r)2^{\omega(r)}}{r^2} \sum_{\substack{a\leq \sqrt{N}/r }} \frac{1}{a} \sum_{\substack{d\leq N^{3/8}/r}} \frac{1}{d} \ll N \sum_{j\leq J} 4^j \sum_{r\geq N^{1/16}} \frac{1}{r^{3/2}} \ll N^{31/32 + \varepsilon}, $$
for any $\varepsilon>0$ very small since $\sum_{j\leq J} 4^j \ll (\log N)^3$. For the remaining $r\leq N^{1/16}$, we use Lemma \ref{fixed_prime_factor_smooth_number_calculation} on the $N^{3/16}/r\leq d\leq N^{3/8}/r $ sum since $d\geq N^{3/16}/r\geq N^{1/8}$, and together with $r\geq N^{1/16}$ contribution we get
\begin{align*}
\eqref{equation_for_case_2} &\ll \frac{N}{(\log N)^2} \sum_{j\leq J} 4^j\sum_{r\leq N^{1/16}} \frac{\mu^2(r)2^{\omega(r)}}{r^2} \sum_{i+\ell \leq K-2^{j+3}} \sum_{\substack{a\leq \sqrt{N}/r \\ \omega(a)=i}}\frac{1}{a} e^{-O(2^j)}\frac{(\log\log N+O(1))^{\ell}}{\ell !} + N^{31/32+\varepsilon}\\
&\ll \frac{N}{(\log N)^2} \sum_{j\leq J} \frac{1}{e^j} \sum_{r\leq N^{1/16}} \frac{\mu^2(r)2^{\omega(r)}}{r^2} \sum_{i+\ell \leq K-2^{j+3}} \frac{(\log \log N+O(1))^{i+\ell}}{i! \ell !} + N^{31/32+\varepsilon}\\
&= \frac{N}{(\log N)^2} \sum_{j\leq J} \frac{1}{e^j} \sum_{r\leq N^{1/16}} \frac{\mu^2(r)2^{\omega(r)}}{r^2} \sum_{k \leq K-2^{j+3}} \frac{(2 \log \log N+O(1))^{k}}{k!} + N^{31/32+\varepsilon}\\
&\ll \frac{N}{(\log N)^{\delta} \sqrt{\log\log N}} \sum_{j\leq J}\frac{1}{e^j (2\log 2+o(1))^{2^{j+3}}} + N^{31/32+\varepsilon}.
\end{align*}
The sum over $j\leq J$ converges, so we get the bound. We only need to check the case $j>J$, but in this case we have $P^+(d)\leq C(N)$. So, bounding trivially we have
$$ \leq \sum_{a\leq \sqrt{N}} \sum_{\substack{d\geq N^{3/16} \\ P^+(d)\leq C(N)}} \sum_{m\leq N/ad} 1 \leq N \sum_{a\leq \sqrt{N}} \frac{1}{a} \sum_{\substack{d\geq N^{3/16} \\ P^+(d)\leq C(N)}} \frac{1}{d} \ll \frac{N}{(\log N)^2},$$
since the sum over $d$ is the tail of a convergent sum and can be bounded by $\ll (\log N)^{-10}$ by Koukoulopoulos \cite{MR3971232}[Theorem 16.3]. So, we get the bound in this case as well. \\
\textit{Case of $P^-(m)\leq N^{3/16}$ and $d\leq N^{3/16}$:} In this case, we trivially bound $r(n)\leq \tau(n)$, then since $n\leq N$, we bound $\tau(n)\ll N^{\varepsilon}$. Now, because we have $p:=P^-(m)\leq N^{3/16}$ and $d\leq N^{3/16}$, we get $N^{3/8} < p^{e} d \leq N^{(1+e)3/16}$ which implies $e> 1$. So we can write $n=p^e n'$,and get in this case trivially bounding, for every $\varepsilon >0$,
\begin{align*}
\ll N^{\varepsilon} \sum_{e\geq 2} \sum_{p^e\geq N^{3/16}} \sum_{n'\leq N/p^e} 1 &\leq N^{1+\varepsilon} \sum_{e\geq 2} \sum_{ p^e\geq N^{3/16}} \frac{1}{p^e} \\
&\leq N^{15/16+\varepsilon} \sum_{e\geq 2} \sum_p \frac{1}{p^{2e/3}} \\
&\ll N^{15/16+\varepsilon},
\end{align*}
where we use $p^{e/3}\geq N^{1/16}$ to get the last line.
\subsection*{Upper bound for the second moment of \texorpdfstring{$r(n)$}{r(n)}:} We first get
$$ 
\sum_{\substack{n\leq N \\ \omega(n)\leq K}} r^2(n) \leq \sum_{\substack{ab=uv\leq N \\ \{a,b\}\not=\{u,v\} \\ \omega(ab)\leq K \\ P^-((a^2+b^2)(u^2+v^2))>N^{1/9} }} 1 + 2\cdot \sum_{\substack{ab\leq N \\ \omega(ab)\leq K \\ P^-(a^2+b^2)>N^{1/9}}} 1,
$$
by looking at the non-diagonal $\{a,b\}\not=\{u,v\}$ and diagonal $\{a,b\}=\{u,v\}$ solutions. The latter sum is the same one we studied in last section. For the former sum, we use the same idea as the latter sum. So we write $a\leq b$, and then split $b=md$ with $P^-(m)>\max\{P^+(d) , (N^{3/8}/d)^{1/e_{\ell+1}}\}$ and $d\leq N^{3/8} < d p_{\ell+1}^{e_{\ell+1}}$. Then we write $d=d_1d_2$ and $m=m_1m_2$ so that $d_1m_1 |u$, and $d_2m_2|v$. Thus, if we let $u'=u/d_1m_1$ and $v'=v/d_2m_2$, then we have a sum $u'v' = a \leq \sqrt{N}$ and sums over smooth $d_1d_2 \leq N^{3/8}$, and rough $m_1m_2\leq N/u'v'd_1d_2$, with $P^-((u'^2v'^2+d_1^2d_2^2m_1^2m_2^2)(u'^2d_1^2m_1^2 + v'^2d_2^2m_2^2))>N^{1/9}$. Moreover, again if $ b\leq N^{3/8}$ that is $m_1m_2=1$, then we have trivially the bound $\ll N^{3/4+\varepsilon}$. So, we can assume, without loss of generality, $m_1>1$, and hence $m_1> P^-(m_1)$. We again split into cases according to the sizes of $P^-(m_1m_2)$ and $d_1d_2$. \\
\textit{Case of $P^-(m_1m_2)>N^{3/16}$:} We have in this case that $m_1m_2\leq N$ and $P^-(m_1), P^-(m_2)>N^{3/16}>N^{1/9}$, hence $\omega(m_1m_2)\leq 16/3$. So, again the prime factors concentrates around $u'v'd_1d_2$. So we can bound by
\begin{align*}
&\ll \sum_{u'v'\leq \sqrt{N}} \sum_{\substack{d_1d_2\leq N^{3/8} \\ \omega(u'v'd_1d_2)\leq K}} \sum_{\substack{m_2\leq N^{13/16}/d_1d_2u'v' \\ P^-(m_2)>N^{3/16}}}\sum_{\substack{N^{3/16}<m_1\leq N/d_1d_2u'v'm_2 \\ P^-(m_1(u'^2d_1^2m_1^2+v'^2d_2^2m_2^2)(u'^2v'^2+d_1^2d_2^2m_1^2m_2^2))>N^{1/9}}} 1 \\
&\ll \frac{N}{(\log N)^3} \sum_{u'v'\leq \sqrt{N}} \sum_{\substack{d_1d_2\leq N^{3/8} \\ \omega(u'v'd_1d_2)\leq K}} \sum_{\substack{m_2\leq N/d_1d_2u'v' \\ P^-(m_2)>N^{3/16}}}\frac{\mathfrak{S}(u',v',d_1,d_2,m_2)}{u'v'd_1d_2m_2}\\
&\ll \frac{N}{(\log N)^3}  \sum_{r} \frac{\mu^2(r)2^{\omega(r)}}{r} \sum_{u'v'\leq \sqrt{N}} \frac{1}{u'v'} \sum_{\substack{d_1d_2\leq N^{3/8} \\ \omega(u'v'd_1d_2)\leq K}} \frac{1}{d_1d_2} \sum_{\substack{m_2\leq N/u'v'd_1d_2 \\ r|u'v'd_1d_2m_2 \\ P^-(m_2)>N^{3/16}}} \frac{1}{m_2} \\
&\leq \frac{N}{(\log N)^3}  \sum_{r} \frac{\mu^2(r)2^{\omega(r)}}{r} \sum_{\substack{h\leq N^{7/8} \\ \omega(h)\leq K }} \frac{\tau_4(h)}{h} \sum_{\substack{m_2\leq N/h \\ r|hm_2 \\ P^-(m_2)>N^{3/16}}} \frac{1}{m_2}, 
\end{align*}
where we used Lemma \ref{second_beta_sieve_result} to get the second line, \eqref{more_variable_singular_series_bound} for the third line, and $h=u'v'd_1d_2$ for the last line. Now, we split $hm_2=rkm$, then we have $\tau_4(h)\leq \tau_4(k) \tau_4(r)$ by semi-multiplicativity of $\tau_4$. Hence, we have
\begin{align*}
\sum_{\substack{h\leq N^{7/8}\\ \omega(h)\leq K }} \frac{\tau_4(h)}{h} \sum_{\substack{m_2\leq N/h \\ r|hm_2 \\ P^-(m_2)>N^{3/16}}} \frac{1}{m_2}&\leq \frac{\tau_4(r)}{r}\sum_{\substack{k\leq N^{7/8} \\ \omega(k)\leq K}} \frac{\tau_4(k)}{k} \sum_{\substack{m\leq N \\ P^-(m)>N^{3/16}}} \frac{1}{m} \\
&\ll \frac{\tau_4(r)}{r} \frac{1}{K!}\bigg( \sum_{p\leq N^{7/8}} \frac{4}{p} +O(1) \bigg)^K \\
&\ll \frac{\tau_4(r)}{r} \cdot \frac{(\log N)^{3-\delta}}{\sqrt{\log\log N}},
\end{align*}
where we bound the sum over $m$ by a constant, and used Erd\H{o}s' trick, Stirling's approximation with $K=\lfloor \frac{\log\log N}{\log 2} \rfloor$. So, similar to the previous subsection, we get the upper bound in this case. \\
\textit{Case of $P^-(m_1m_2)\leq N^{3/16}$ and $d_1d_2>N^{3/16}$:} Let $z=N^{3/8}$. We again do a dyadic decomposition $z_j \geq P^+(d_1d_2) > z_{j+1}$ with $z_j=z^{2^{-j}}$, and we choose $J$ such that $z_{J+1} \leq C(N) < z_J$. Thus, we have $P^-(m_1m_2)>z_{j+1}$ and $P^-((u'^2v'^2+m_1^2m_2^2d_1^2d_2)(u'^2d_1^2m_1^2+v'^2d_2^2m_2^2))>N^{1/9}$ can be bounded by $P^-((u'^2v'^2+m_1^2m_2^2d_1^2d_2^2)(u'^2d_1^2m_1^2+v'^2d_2^2m_2^2))>z_{j+1}$ since $z_{j+1} \leq z_2 \leq N^{1/9}$. Notice also that since $P^-(m_1m_2)>z^{2^{-j-1}}$, we have $\omega(m_1m_2)\leq 2^{j+4}/3\leq 2^{j+3}$, and since $(u'v'd_1d_2,m_1m_2)=1$, we have $\omega(u'v'd_1d_2)\leq K- 2^{j+3}$. We again recall $2^{j+3} =o(\log\log N)$ for $j\leq J$. Similarly, we get
\begin{align}
&\sum_{j\leq J} \sum_{u'v'\leq \sqrt{N}} \sum_{\substack{N^{3/16}\leq d_1d_2\leq N^{3/8} \\ P^+(d_1d_2)\leq z_{j+1} \\ \omega(ad)\leq K-2^{j+3}}} \sum_{\substack{m_2\leq N/u'v'd_1d_2 \\ P^-(m_2)>z_{j+1}}}\sum_{\substack{m_1\leq N/u'v'd_1d_2m_2 \\ P^-(m_1(u'^2v'^2+m_1^2m_2^2d_1^2d_2^2)(u'^2d_1^2m_1^2+v'^2d_2^2m_2^2))>z_{j+1}}} 1 \nonumber \\
&\ll \frac{N}{(\log N)^3} \sum_{j\leq J} 8^j \sum_{u'v'\leq \sqrt{N}} \sum_{\substack{N^{3/16}\leq d_1d_2\leq N^{3/8} \\ P^+(d_1d_2)\leq z_{j+1} \\ \omega(u'v'd_1d_2)\leq K-2^{j+3}}} \sum_{\substack{m_2\leq N/u'v'd_1d_2 \\ P^-(m_2)>z_{j+1}}} \frac{\mathfrak{S}(u',v',d_1,d_2,m_2)}{u'v'd_1d_2m_2} \nonumber \\
&\ll \frac{N}{(\log N)^3}  \sum_{j \leq J} 8^j \sum_{r} \frac{\mu^2(r)2^{\omega(r)}}{r} \sum_{u'v'\leq \sqrt{N}} \frac{1}{u'v'} \sum_{\substack{N^{3/16}\leq d_1d_2\leq N^{3/8} \\ P^+(d_1d_2)\leq z_{j+1} \\ \omega(u'v'd_1d_2)\leq K-2^{j+3}}} \frac{1}{d_1d_2} \sum_{\substack{m_2\leq N/u'v'd_1d_2 \\ r|u'v'd_1d_2m_2 \\ P^-(m_2)>z_{j+1}}} \frac{1}{m_2} \nonumber \\
&\leq  \frac{N}{(\log N)^3}  \sum_{j\leq J} 8^j \sum_{r} \frac{\mu^2(r)2^{\omega(r)}}{r^2}\sum_{i+\ell \leq K- 2^{j+3}} \sum_{\substack{u'v'\leq \sqrt{N}/r \\ \omega(u'v')=i}} \frac{1}{u'v'} \sum_{\substack{N^{3/16}/r\leq d_1d_2\leq N^{3/8}/r \\ P^+(d_1d_2)\leq z_{j+1} \\ \omega(d_1d_2)=\ell}} \frac{1}{d_1d_2} \sum_{\substack{m_2\leq N/u'v'd_1d_2r \\ P^-(m_2)>z_{j+1}}} \frac{1}{m_2}, \label{equation_for_case_2_more_variable}
\end{align}
where we can split $\omega(u'v'd_1d_2)=\omega(u'v')+\omega(d_1d_2)$ because $\gcd(u'v',d_1d_2)=1$ and for the last line we took out the $r$. We again notice that we can bound the contribution of $r\geq N^{1/16}$ trivially,
\begin{align*}
&\frac{N}{(\log N)^3}   \sum_{j\leq J} 8^j\sum_{r\geq N^{1/16}} \frac{\mu^2(r)2^{\omega(r)}}{r^2} \sum_{\substack{u'v'\leq \sqrt{N}/r }} \frac{1}{u'v'} \sum_{\substack{d_1d_2\leq N^{3/8}/r}} \frac{1}{d_1d_2} \sum_{m_2\leq N} \frac{1}{m_2} \\
&\ll N \sum_{j\leq J} 8^j \sum_{r\geq N^{1/16}} \frac{1}{r^{3/2}} \ll N^{31/32 + \varepsilon},
\end{align*}
for any $\varepsilon>0$ very small since $\sum_{j\leq J} 8^j \ll (\log N)^4$. Now, for the remaining $r\leq N^{1/16}$, we first bound the sum over $m_2$. To do this notice that we can bound it by partial summation
$$\sum_{\substack{m_2\leq N \\ P^-(m_2)>z_{j+1}}} \frac{1}{m_2} \ll \frac{2^j}{\log N}+ \int_{z_{j+1}}^N \frac{dt}{t(\log z_{j+1})} \ll 2^j.$$
Thus, if we let $d=d_1d_2$ and $n=u'v'$, then we have
$$ \eqref{equation_for_case_2_more_variable} \ll   \frac{N}{(\log N)^3}  \sum_{j\leq J} 16^j \sum_{r} \frac{\mu^2(r)2^{\omega(r)}}{r^2}\sum_{i+\ell \leq K- 2^{j+3}} \sum_{\substack{n\leq \sqrt{N}/r \\ \omega(n)=i}} \frac{\tau(n)}{n} \sum_{\substack{N^{3/16}/r\leq d\leq N^{3/8}/r \\ P^+(d)\leq z_{j+1} \\ \omega(d)=\ell}} \frac{\tau(d)}{d} + N^{31/32+\varepsilon} .$$
Now, we use Lemma \ref{fixed_prime_factor_smooth_number_calculation} for the sum, since $d\geq N^{3/16}/r\geq N^{1/8}$, to get
\begin{align*}
&\leq \frac{N}{(\log N)^3} \sum_{j\leq J} 16^j\sum_{r\leq N^{1/16}} \frac{\mu^2(r)2^{\omega(r)}}{r^2} \sum_{i+\ell \leq K-2^{j+3}} \sum_{\substack{n\leq \sqrt{N}/r \\ \omega(n)=i}}\frac{\tau(n)}{n} e^{-O(2^j)}\frac{(2\log\log N + O(1))^{\ell}}{\ell !}  \\
&\ll \frac{N}{(\log N)^3} \sum_{j\leq J} \frac{1}{e^j} \sum_{r\leq N^{1/16}} \frac{\mu^2(r)2^{\omega(r)}}{r^2} \sum_{i+\ell \leq K-2^{j+3}} \frac{(2\log \log N+O(1))^{i+\ell}}{i! \ell !} \\
&= \frac{N}{(\log N)^3} \sum_{j\leq J} \frac{1}{e^j} \sum_{r\leq N^{1/16}} \frac{\mu^2(r)2^{\omega(r)}}{r^2} \sum_{k \leq K-2^{j+3}} \frac{(4 \log \log N+O(1))^{k}}{k!} \\
&\ll \frac{N}{(\log N)^{\delta} \sqrt{\log\log N}} \sum_{j\leq J}\frac{1}{e^j (4\log 2+o(1))^{2^{j+3}}}.
\end{align*}
The sum over $j\leq J$ converges, so we get the bound. We only need to check the case when $j>J$, that is $P^+(d_1d_2)\leq C(N)$. But in this case bounding trivially, we have
\begin{align*}
\leq \sum_{u'v'\leq \sqrt{N}} \sum_{\substack{d_1d_2\geq N^{3/16} \\ P^+(d_1d_2)\leq C(N)}} \sum_{m_1m_2\leq N/ad} 1 &\leq N \log N \sum_{u'v'\leq \sqrt{N}} \frac{1}{u'v'} \sum_{\substack{d_1d_2\geq N^{3/16} \\ P^+(d_1d_2)\leq y}} \frac{1}{d_1d_2} \\
&\leq N \log N \sum_{n\leq \sqrt{N}} \frac{\tau(n)}{n} \sum_{\substack{d\geq N^{3/16} \\ P^+(d)\leq C(N)}} \frac{\tau(d)}{d} \ll \frac{N}{(\log N)^2},
\end{align*}
since the sum over $d$ is the tail of a convergent sum, and since $C(N)$ is as defined in \eqref{cut_off_variable}, it can be bounded by $\ll (\log N)^{-10}$ by Koukoulopoulos \cite{MR3971232}[Theorem 16.3]. So, we get the bound in this case as well. \\
\textit{Case of $P^-(m_1m_2)\leq N^{3/16}$ and $d_1d_2\leq N^{3/16}$:} In this case, we again trivially bound $r^2(n)\leq \tau^2(n)$, then since $n\leq N$, we bound $\tau^2(n)\ll N^{\varepsilon}$. Now, because we have $p:=P^-(m_1m_2)\leq N^{3/16}$ and $d_1d_2\leq N^{3/16}$, we get $N^{3/8} < p^{e} d_1d_2 \leq N^{(1+e)3/16}$ which implies $e> 1$. So we can write $n=p^e n'$,and get in this case the same result as the average of $r(n)$, for every $\varepsilon >0$,
\begin{align*}
\ll N^{\varepsilon} \sum_{e\geq 2} \sum_{p^e\geq N^{3/16}} \sum_{n'\leq N/p^e} 1 &\leq N^{1+\varepsilon} \sum_{e\geq 2} \sum_{ p^e\geq N^{3/16}} \frac{1}{p^e} \\
&\leq N^{15/16+\varepsilon} \sum_{e\geq 2} \sum_p \frac{1}{p^{2e/3}} \\
&\ll N^{15/16+\varepsilon},
\end{align*}
where we use $p^{e/3}\geq N^{1/16}$ to get the last line.
\section{Lower Bound}\label{section_lower_bound}
We will show a lower bound for the average of $r(n)$. Let $n$ be an element in $\CB(N)$, then $n=ab$ for some $a,b\in \N$ such that $\Omega(a^2+b^2)\leq 5$ and $P^-(a^2+b^2)>N^{1/9}$. For the lower bound, we look at the subset of $n=abp$ where $ab\leq N^{1/1000}$ and $p \leq N/ab$. We also force $ab$ to be square-free to have the $\gcd(a,b)=1$ condition to satisfy the admissibility. Thus, we have
\begin{align}
\sum_{\substack{n\leq N \\ \omega(n)\leq K}}r(n) &\geq \sum_{\substack{ab\leq N^{1/1000} \\ \gcd(a,b)=1 \\ \omega(ab)= K}} \mu^2(ab) \sum_{\substack{p \leq N/ab \\ \Omega(a^2+p^2b^2)\leq 5 \\ P^-(a^2+p^2b^2)>N^{1/9}}} 1 \nonumber \\
&\geq \sum_{\substack{ab\leq N^{1/1000} \\ \gcd(a,b)=1 \\ \omega(ab)= K}} \mu^2(ab) \sum_{\substack{p \leq N/ab \\ \Omega(a^2+p^2b^2)\leq 5 \\ P^-(a^2+p^2b^2)>\left(\frac{N}{ab(\log N/ab)}\right)^{1/8}}} 1. \label{main_lower_sum}
\end{align}
Now, for the inner sum in \eqref{main_lower_sum} we use the weighted sieve result Lemma \ref{weighted_sieve_result} and $1\leq ab\leq N^{1/1000}$ to get
$$\sum_{\substack{n\leq N \\ \omega(n)= K}}r(n) \gg \frac{N}{(\log N)^2} \sum_{\substack{ab\leq N^{1/1000} \\ \omega(ab)= K}} \mu^2(ab) \frac{\mathfrak{S}'(a,b)}{ab} .$$
Now, we study the singular series.
\begin{lem}\label{second_singular_series_bound}
For any $a,b\in \N$ with $2| ab$, we have
$$ \mathfrak{S}'(a,b)\gg 1 .  $$
\end{lem}
\begin{proof}
We first recall the singular series,
$$ \mathfrak{S}'(a,b) = \prod_{p\nmid ab} \bigg( 1 - \frac{1+\chi_4(p)}{p-1}\bigg) \bigg( 1  -\frac{1}{p}\bigg)^{-1} \prod_{p|ab} \bigg( 1 -  \frac{1}{p}\bigg)^{-1}. $$
We take out $p=2$ from the product over $p|ab$, and complete the sum over $p\nmid ab$ to get
$$ \mathfrak{S}'(a,b)\asymp \prod_{\substack{p>2 \\ p|ab}} \bigg(1 - \frac{1+\chi_4(p)}{p-1} \bigg)^{-1} \geq 1. $$
\end{proof}
Then by Lemma $\ref{second_singular_series_bound}$, we have
\begin{align}
\frac{N}{(\log N)^2} \sum_{\substack{ab\leq N^{1/1000}  \\ \omega(ab)= K}} \mu^2(ab) \frac{\mathfrak{S}'(a,b)}{ab} &\gg \frac{N}{(\log N)^2} \sum_{\substack{ab\leq N^{1/1000} \\ 2|ab  \\ \omega(ab)= K}} \frac{\mu^2(ab)}{ab} \nonumber \\
&\gg \frac{N}{(\log N)^2} \sum_{\substack{a_0b_0 \leq N^{1/1000}/2 \\ 2\nmid a_0b_0 \\ \omega(a_0b_0)= K-1}} \frac{\mu^2(a_0b_0)}{a_0b_0} \label{main_lower_bound}
\end{align}
Now, let $n=a_0b_0$, then we have
$$ \sum_{\substack{n\leq N^{1/1000}/2 \\ 2\nmid n \\ \omega(n)= K-1}} \frac{\mu^2(n) \tau(n)}{n} = \frac{(2\log\log N + O(1))^{K-1}}{(K-1)!}, $$
by Lemma \ref{almost_prime_on_divisor_function}. Thus, we have for \eqref{main_lower_bound},
$$ \gg \frac{N}{(\log N)^{\delta} \sqrt{\log\log N}},$$
by Stirling's approximation with $K=\lfloor \frac{\log\log N}{\log 2} \rfloor $. Therefore, we get
$$ \sum_{\substack{n\leq N \\ \omega(n)\leq K}} r(n) \gg \frac{N}{(\log N)^{\delta} \sqrt{\log\log N}}. $$

\nocite{*}
\bibliographystyle{plain}
\bibliography{bibliography_for_almost_prime_hypotenuse}

\end{document}